\documentclass[12pt]{article}                     
\usepackage{titlesec,a4}
\usepackage{graphicx,epsfig}
\usepackage{amsfonts,color,morefloats,pslatex}
\usepackage{amssymb, amsmath,latexsym,amsthm}

\newtheorem{theorem}{Theorem}

\newtheorem{corollary}[theorem]{Corollary}
\theoremstyle{definition}
\newtheorem{example}[theorem]{Example}
\newtheorem{conj}[theorem]{Conjecture}
\newtheorem{note}[theorem]{Note}

\titleformat{\section}
  {\normalfont\fontsize{12}{12}\bfseries}{\thesection}{1em}{} 
  
\newcommand{\tr}{{\mathrm{Tr}}}
\newcommand{\Tr}{{\mathrm{Tr}}}

\newcommand{\gf}{{\mathrm{GF}}}

\newcommand{\RM}{{\mathrm{RM}}}

\newcommand{\cP}{{\mathcal{P}}}
\newcommand{\cB}{{\mathcal{B}}}

\newcommand{\C}{{\mathcal{C}}}

\newcommand{\bD}{{\mathbb{D}}}

\begin{document}

\title{Bent Vectorial Functions, Codes and Designs} 


\author{Cunsheng Ding\footnote{Department of Computer Science and Engineering, The Hong Kong University of Science and Technology, Hong Kong. Email: cding@ust.hk}, Akihiro Munemasa\footnote{Research Center for Pure and Applied Mathematics, Graduate School of Information Sciences, Tohoku University, Sendai 980-8579, Japan. Email: munemasa@math.is.tohoku.ac.jp}, 
Vladimir D. Tonchev\footnote{Department of Mathematical Sciences, Michigan Technological University, Houghton, MI 49931, USA. Email: tonchev@mtu.edu}}



\maketitle

\begin{abstract}

Bent functions, or equivalently, Hadamard difference sets in the
elementary Abelian group 
$(\gf(2^{2m}), +)$, 
have been employed to construct  symmetric and quasi-symmetric 
designs having the symmetric difference property
\cite{Kantor75}, \cite{DS87}, \cite{Kantor83}, \cite{JT91}, \cite{JT92}. 
The main objective of this paper is to use bent vectorial functions 
for a construction of a two-parameter family of binary linear 
codes that do not satisfy the conditions of the Assmus-Mattson
theorem,
but nevertheless hold $2$-designs. 
A new coding-theoretic characterization of bent vectorial functions is presented.  

\end{abstract}



\section{Introduction, motivations and objectives}\label{sec-Introd}

We start with a brief review of combinatorial $t$-designs (cf. \cite{AK92}, \cite{BJL}, \cite{T88}).
Let $\cP$ be a set of $v \ge 1$ elements, called {\it points},
 and let $\cB$ be a collection of $k$-subsets of $\cP$, called {\it blocks}, where $k$ is
a positive integer, $1 \leq k \leq v$. Let $t$ be a non-negative integer,  
$t \leq k$. The pair
$\bD = (\cP, \cB)$ is called a $t$-$(v, k, \lambda)$ {\em design\index{design}}, or simply {\em $t$-design\index{$t$-design}}, if every $t$-subset of $\cP$ is contained in exactly $\lambda$ blocks of
$\cB$. 
We usually use $b$ to denote the number of blocks in $\cB$.  A $t$-design is called {\em simple\index{simple}} if $\cB$ does not contain any repeated blocks. In this paper, we consider only simple  $t$-designs.  

Two designs are {\it isomorphic} if there is a bijection between their
point sets that maps every block of the first design to a block of the second
design. An {\it automorphism} of a design is any isomorphism of the design
to itself.
The set of all automorphisms  of a design $\bD$ form the (full)  
automorphism group 
of $\bD$.

It is clear that $t$-designs with $k = t$ or $k = v$ always exist. Such $t$-designs are called {\em trivial}. In this paper, we consider only $t$-designs with $v > k > t$.

The incidence matrix of a design $\bD$ is a $(0,1)$-matrix $A=(a_{ij})$ with rows
labeled by the blocks, columns labeled by the points, where
$a_{i,j}=1$ if the $i$th block contains the $j$th point, and $a_{i,j}=0$
otherwise.
If the incidence matrix is viewed over $\gf(q)$, 
its rows span a linear code of length $v$ over $\gf(q)$, which is denoted 
by $\C_q(\bD)$ and is called the code of the design. Note that a $t$-design can be employed to 
construct linear codes in different ways. The supports of  codewords of a 
given Hamming weight $k$ in a code $\C$ may form a $t$-design, which is referred to 
as a  design supported by the code. 

A design is called {\em symmetric\index{symmetric design}} if $v = b$.
A $2$-$(v, k, \lambda)$ design is symmetric if and only if every two blocks
share exactly $\lambda$ points.

A $2$-design is \emph{quasi-symmetric}\index{quasi-symmetric} with intersection numbers $x$ and $y$, ($x < y$) if any two blocks 
intersect in either $x$ or $y$ points.

Let $\bD=\{\cP, \,\cB \}$ be a $2$-$(v, k, \lambda)$ symmetric design, where $\cB=\{B_1,\, B_2,\, \cdots, \, B_v\}$ 
and $v \geq 2$. 
Then 
\begin{itemize}
\item $(B_1, \, \{B_2 \cap B_1,\, B_3 \cap B_1,\, \cdots,\, B_v \cap B_1 \})$ 
is a $2$-$(k,\, \lambda,\, \lambda-1)$ design, and called the \emph{derived design}\index{derived design} of $\bD$ with respect to $B_1$; 
\item $(\overline{B}_1,\, \{B_2 \cap \overline{B}_1,\, B_3 \cap \overline{B}_1,\, \cdots, B_v \cap \overline{B}_1 \})$ 
is a $2$-$(v-k,\, k-\lambda,\, \lambda)$ design, called the \emph{residual design}\index{residual design} of $\bD$ with respect to $B_1$, 
where $\overline{B_1}=\cP \setminus B_1$. 
\end{itemize}

If a symmetric design $\bD$ has parameters 
\begin{eqnarray}\label{eqn-SDPparameters} 
2-(2^{2m},\, 2^{2m-1}-2^{m-1},\, 2^{2m-2}-2^{m-1}),  
\end{eqnarray}
its derived designs
have parameters 
\begin{eqnarray*}  
2-( 2^{2m-1}-2^{m-1},\, 2^{2m-2}-2^{m-1},\, 2^{2m-2}-2^{m-1}-1), 
\end{eqnarray*}
and its residual designs have parameters 
\begin{eqnarray*}  
2-( 2^{2m-1}+2^{m-1},\, 2^{2m-2},\, 2^{2m-2}-2^{m-1}).  
\end{eqnarray*} 

A symmetric $2$-design is said to have the \emph{symmetric difference property}, 
or to be a  
\emph{symmetric SDP design}\index{symmetric SDP design},
(Kantor \cite{Kantor75, Kantor83}),  if the symmetric difference of 
any \emph{three} blocks is either a block or the complement of a block. 
Any derived or residual design of a symmetric SDP design
is quasi-symmetric, and has the property that the symmetric
difference of every two blocks is either a block or the complement
of a block. The derived and residual designs
of a symmetric SDP design are called quasi-symmetric SDP designs \cite{JT92}.
The binary codes of quasi-symmetric SDP designs 
give rise to an exponentially growing number of inequivalent
linear codes that meet the Grey-Rankin bound \cite{JT91}.
It was proved in \cite{Tonchev93} that any quasi-symmetric
SDP design can be embedded as a derived or a residual
design in exactly one (up to isomorphism) symmetric SDP design.

A coding-theoretical characterization of symmetric SDP
designs was given by
Dillon and Schatz \cite{DS87}, who proved that any symmetric
SDP design with parameters (\ref{eqn-SDPparameters})
is supported by the codewords of minimum weight
in a binary linear code $\C$ of length $2^{2m}$, dimension
$2m+2$ and weight enumerator given by
\begin{equation}
\label{sdpcode}
1 + 2^{2m}z^{2^{2m-1}-2^{m-1}} + (2^{2m+1} - 2)z^{2^{2m-1}}+ 2^{2m}z^{2^{2m-1}+2^{m-1}} + z^{2m}, 
\end{equation}
where $\C$ is spanned by the first order Reed-Muller code
$\RM_2(1, 2m)$
and a vector $u$ being the truth table (introduced in Section~\ref{sec-333})
of a bent function
in $2m$ variables, or equivalently, $u$ is the incidence vector
of a Hadamard difference set in the additive group of $\gf(2)^{2m}$
 with parameters
\begin{equation*}
\label{eqn-MenonHadamardPara}
(2^{2m}, \,2^{2m-1} \pm 2^{m-1}, \,2^{2m-2} \pm 2^{m-1}).  
\end{equation*}  

One of the objectives  of this paper is to give a coding-theoretical 
characterization of bent vectorial functions (Theorem \ref{thm-bentvectf}),
which generalizes the Dillon and Schatz characterization of
single bent functions \cite{DS87}.
  Another objective
 is to present
in Theorem \ref{main}
 a two-parameter family of 
 binary linear codes 
with parameters 
\[  [2^{2m},2m+1 +\ell,2^{2m-1} - 2^{m-1}],  \ m \ge 2, \ 1\le \ell \le m,  \]
that are based on bent vectorial functions
and support $2$-designs, despite that 
these codes do not satisfy the conditions
of the Assmus-Mattson theorem (see Theorem \ref{thm-AM1}).
The subclass of codes with $\ell=1$ consists of 
codes introduced by Dillon and Schatz \cite{DS87} that are based on bent functions
and support symmetric SDP designs. 
Examples of codes  with $\ell=m$ are given that are optimal
in the sence that they have the maximum possible minimum distance
for the given length and dimension, or 
have the largest known minimum distance for the given length and dimension
(see Note \ref{Note6} in Section \ref{Section4}, and the examples thereafter).

\section{The classical constructions of $t$-designs from codes}

A simple sufficient condition for the supports of 
codewords of any given weight in a linear code
to support a $t$-design is that the code admits
a $t$-transitive or $t$-homogeneous automorphism
group.
All codes considered in this paper are of even length $n$ of the form
$n=2^{2m}$. It is known that  any 2-homogeneous group of even degree is
necessarily 2-transitive (Kantor \cite{K69, K85}).

Another  sufficient condition is given by the  Assmus-Mattson theorem.
Let $\C$ be a $[v, \kappa, d]$ linear code over $\gf(q)$, and let $A_i =A_i(\C)$
 be the
number of codewords of Hamming weight $i$ in $\C$  ($0 \leq i \leq v$). 
For each $k$ with $A_k \neq 0$,  let $\cB_k$ denote
the set of the supports of all codewords of Hamming weight $k$ in $\C$, 
where the code coordinates 
are indexed by $1,2, \ldots, v$. Let $\cP=\{ 1, 2, \ldots, v \}$.  
The following
theorem, proved by Assumus and Mattson, provides  sufficient  conditions
 for the pair $(\cP, \cB_k)$  to be a $t$-design.

\begin{theorem}[The Assmus-Mattson Theorem \cite{AM69}]\label{thm-AM1}
Let $\C$ be a binary $[v, \kappa, d]$ code, and  let   $d^\perp$  be the minimum 
weight of the dual code $\C^\perp$.
Suppose that $A_i=A_i(\C)$ and $A_i^\perp=A_i(\C^\perp)$,  $0 \leq i \leq v$, are the
weight distributions of $\C$ and $\C^\perp$, respectively. Fix a positive integer $t$
with $t < d$, and let $s$ be the number of $i$ with $A_i^\perp \ne 0$ for $0 < i \leq v-t$.
If $s \leq d -t$, then
\begin{itemize}
\item the codewords of weight $i$ in $\C$ hold a $t$-design provided that $A_i \ne 0$ and
      $d \leq i \leq v$, and
\item the codewords of weight $i$ in the  code $\C^\perp$ hold a $t$-design provided that
      $A_i^\perp \ne 0$ and $d^\perp \leq i \leq v-t$.
\end{itemize}
\end{theorem}

The parameter $\lambda$ of a $t$-$(v,w,\lambda)$ design 
supported by the codewords  of weight $w$ in a binary code  $\C$  is determined by 
\begin{equation*}
  A_w = \lambda { v \choose t}/{w \choose t}. 
  \end{equation*}

\section{Bent functions and bent vectorial functions}\label{sec-333} 

Let $f=f(x)$ be a Boolean function from $\gf(2^{n})$ to $\gf(2)$. The \emph{support} 
$S_f$ of $f$ is defined 
as
$$
S_f=\{x \in\gf(2^{n}) : f(x)=1\} \subseteq \gf(2^{n}). 
$$ 
The $(0,1)$ incidence vector of $S_f$, having its coordinates labeled
by the elements of  $\gf(2^n)$,  is called the {\it truth table} of $f$.

The {\em Walsh transform} of $f$ is defined by 
\begin{eqnarray*}
\hat{f}(w)=\sum_{x \in \gf(2^{n})} (-1)^{f(x)+\tr_{n/1}(wx)} 
\end{eqnarray*} 
where $w \in \gf(2^{n})$ and $\tr_{n/n'}(x)$ denotes the trace function from 
$\gf(2^n)$ to $\gf(2^{n'})$.  

Two Boolean functions $f$ and $g$ from $\gf(2^n)$ to $\gf(2)$ are called 
\emph{weakly affinely equivalent} or \emph{EA-equivalent} if there are an 
automorphism $A$ of $(\gf(2^n), +)$, a homomorphism $L$ from $(\gf(2^n),+)$ 
to $(\gf(2), +)$, an element $a \in \gf(2^n)$ and an element $b \in \gf(2)$ such 
that 
$$ 
g(x)=f(A(x)+a)+ L(x) +b 
$$  
for all $x \in \gf(2^n)$.

A  Boolean function $f$ from $\gf(2^{2m})$ to $\gf(2)$ is called a \emph{bent\index{bent}} 
function if
 $|\hat{f}(w)|=
2^{m}$ for every $w \in \gf(2^{2m})$.  
It is well known that 
a function $f$ from $\gf(2^{2m})$ to $\gf(2)$ is bent if and only if $S_f$ is 
a difference set in  $(\gf(2^{2m}),\,+)$ with  parameters (\ref{eqn-MenonHadamardPara}) 
\cite{Mesnagerbook}.

A  Boolean function $f$  from $\gf(2^{2m})$ to $\gf(2)$ is a bent function
if and only if its truth table is at Hamming distance
$2^{2m-1} \pm 2^{m-1}$ from every codeword of the
first order Read-Muller code $\RM_2(1, 2m)$ 
\cite[Theorem 6, page 426]{McS}.
It follows that 
\begin{eqnarray*} 
|S_f|=2^{2m-1} \pm 2^{m-1}. 
\end{eqnarray*} 

There are many constructions of bent functions. The reader is referred to \cite{CMSurvey} 
and \cite{Mesnagerbook} for detailed information about bent functions.

Let $\ell$ be a positive integer, and let $f_1(x), \cdots, f_\ell(x)$ be Boolean functions 
from $\gf(2^{2m})$ to $\gf(2)$. The function $F(x)=(f_1(x), \cdots, f_\ell(x))$ 
from $\gf(2^{2m})$ to $\gf(2)^\ell$
is called a 
$(2m,\ell)$ {\it vectorial} Boolean function. 

A $(2m,\ell)$ vectorial Boolean function $F(x)=(f_1(x), \cdots, f_\ell(x))$ is called a \emph{bent vectorial function} 
if $\sum_{j=1}^\ell a_j f_j(x)$ is a bent function for each nonzero $(a_1, \cdots, a_\ell) 
\in \gf(2)^\ell$. 

For another equivalent definition of bent vectorial functions, 
see \cite{CMbentvect} or \cite[Chapter 12]{Mesnagerbook}.

Bent vectorial functions exist only when $\ell \leq m$ (cf. \cite[Chapter 12]{Mesnagerbook}).
There are a number of known constructions 
of bent vectorial functions. The reader is referred to \cite{CMbentvect} and 
\cite[Chapter 12]{Mesnagerbook} for detailed information. Below we present a specific 
construction of bent vectorial functions from \cite{CMbentvect}.  

\begin{example}\label{exam-bentvectfunc1} 
\cite{CMbentvect}.
Let $m \geq 1$ be an odd
 integer,  $\beta_1, \beta_2, \cdots, \beta_{m}$ be a basis of 
$\gf(2^{m})$ over $\gf(2)$, and let $u \in \gf(2^{2m}) \setminus \gf(2^m)$. Let $i$ 
be a positive integer with $\gcd(2m, i)=1$. Then 
\begin{eqnarray*}
\left(\tr_{2m/1}(\beta_1 u x^{2^i+1}), \tr_{2m/1}(\beta_2 u x^{2^i+1}), \cdots, 
\tr_{2m/1}(\beta_{m} u x^{2^i+1})  \right) 
\end{eqnarray*} 
is a $(2m, m)$ bent vectorial function. 
\end{example} 

Under a basis of $\gf(2^{\ell})$ over $\gf(2)$, $(\gf(2^\ell), +)$ and $(\gf(2)^\ell, +)$ 
are isomorphic. Hence, any vectorial function $F(x)=(f_1(x), \cdots, f_\ell(x))$ from 
$\gf(2^{2m})$ to $\gf(2)^\ell$ can be viewed as a function from $\gf(2^{2m})$ to 
$\gf(2^\ell)$.   

It is well known that a function $F$ from $\gf(2^{2m})$ to $\gf(2^\ell)$ is bent if and 
only if $\tr_{\ell/1}(aF(x))$ is a bent Boolean function for all $a \in \gf(2^\ell)^*$. 
Any such vectorial function $F$ can be expressed as $\tr_{2m/\ell}(f(x))$, where $f$ is 
a univariate polynomial.   
This presentation of bent vectorial functions is more compact. We give two examples  
of bent vectorial functions in this form. 

\begin{example}\label{exam-bentvectfunc2} 
(cf. \cite[Chapter 12]{Mesnagerbook}).
Let $m>1$ and $i \geq 1$ be integers such that $2m/\gcd(i, 2m)$ is even. 
Then $\tr_{2m/m}(a x^{2^i+1})$ is bent if and only if 
$\gcd(2^i+1, 2^m+1) \neq 1$ and 
$a \in \gf(2^{2m})^* \setminus \langle \alpha^{\gcd(2^i+1, 2^m+1)} \rangle$, 
where $\alpha$ is a generator of $\gf(2^{2m})^*$.  
\end{example}

\begin{example}\label{exam-bentvectfunc3} 
(cf. \cite[Chapter 12]{Mesnagerbook}).
Let $m>1$ and $i \geq 1$ be integers such that $\gcd(i, 2m)=1$. Let $d=2^{2i}-2^i+1$. 
Let $m$ be odd. 
Then $\tr_{2m/m}(a x^{d})$ is bent if and only if 
$a \in \gf(2^{2m})^* \setminus \langle \alpha^{3} \rangle$, 
where $\alpha$ is a generator of $\gf(2^{2m})^*$.  
\end{example}   

\section{A construction of codes from bent vectorial functions} 
\label{Section4}

Let $q=2^{2m}$,  let 
$\gf(q)=\{u_1, u_2, \cdots, u_{q}\}$,  and let $w$ be a generator of $\gf(q)^*$. 
For the purposes of what follows, it is convenient to use the following 
generator matrix  of the binary $[2^{2m}, 2m+1,2^{2m-1}]$
 first-order Reed-Muller code $\RM_2(1,2m)$:
\begin{eqnarray*}
G_0=\left[ 
\begin{array}{cccc}
1   &  1   & \cdots  & 1 \\
\tr_{2m/1}(w^0u_1)  & \tr_{2m/1}(w^0u_2)  & \cdots & \tr_{2m/1}(w^0u_q)  \\
\vdots  & \vdots  & \ddots  & \vdots \\
\tr_{2m/1}(w^{2m-1}u_1)  & \tr_{2m/1}(w^{2m-1}u_2)  & \cdots & \tr_{2m/1}(w^{2m-1}u_q)  
\end{array} 
\right].  
\end{eqnarray*} 
The weight enumerator of $\RM_2(1,2m)$ is
\begin{equation}\label{rm1} 
1+(2^{2m+1}-2)z^{2^{2m-1}} + z^{2^{2m}}.
\end{equation} 
Two binary linear codes are equivalent if there is a permutation of coordinates 
that sends the first code to the second. 
Up to equivalence, $\RM_2(1,2m)$ is the unique linear binary
code with parameters  $[2^{2m}, 2m+1,2^{2m-1}]$ \cite{DS87}.
Its dual code is the $[2^{2m}, 2^{2m}- 1 -2m,4]$
Reed-Muller code of order $2m-2$.
Both codes hold 3-designs since they are invariant
under a 3-transitive affine group.
Note that  $\RM_2(1,2m)^\perp$ is 
the unique, up to equivalence, binary linear code
for the given parameters, hence  it is equivalent to the
extended binary linear Hamming code.

Let $F(x)=(f_1(x), f_2(x), \cdots, f_\ell(x))$ be a $(2m, \ell)$ vectorial function from $\gf(2^{2m})$ 
to $\gf(2)^\ell$. For each $i$,  $1 \le i \le \ell$, 
we define a binary vector
\begin{eqnarray*}
F_i=(f_i(u_1), f_i(u_2), \cdots, f_i(u_q)) \in \gf(2)^{2^{2m}}, 
\end{eqnarray*}   
which is the truth table of the Boolean function $f_i(x)$ introduced in Section \ref{sec-333}.

Let $\ell$ be an integer in the range $1 \le \ell \le m$.
We now define a $(2m+1 + \ell) \times 2^{2m}$ matrix 
\begin{eqnarray}
\label{G}
G=G(f_{1}, \cdots, f_{\ell})=\left[ 
\begin{array}{c}
G_0 \\
F_{1} \\
\vdots \\
F_{\ell} 
\end{array} 
\right], 
\end{eqnarray}
where $G_0$ is the generator matrix of $\RM_2(1,2m)$.
Let $\C(f_{1}, \cdots, f_{\ell})$ denote the binary code of length $2^{2m}$ with generator matrix $G(f_{1}, \cdots, f_{\ell})$ given by (\ref{G}).  The dimension of the code has the following lower and upper bounds: 
$$ 
2m+1 \leq \dim(\C(f_{1}, \cdots, f_{\ell})) \leq 2m+1+\ell. 
$$ 

The following theorem gives a coding-theoretical characterization of  bent vectorial functions. 

\begin{theorem}\label{thm-bentvectf} 
A  $(2m, \ell)$ vectorial function $F(x)=(f_1(x), f_2(x), \cdots, f_\ell(x))$ from $\gf(2^{2m})$ to $\gf(2)^\ell$ is a bent vectorial function if and only if the  code  
$\C(f_1, \cdots, f_\ell)$   with generator matrix $G$ given by  (\ref{G}) has 
weight enumerator 
\begin{eqnarray}\label{eqn-wtenumerator111}
1 + (2^\ell-1)2^{2m} z^{2^{2m-1} - 2^{m-1}} + 2(2^{2m}-1)z^{2^{2m-1}} 
+ (2^\ell-1)2^{2m} z^{2^{2m-1} + 2^{m-1}} + z^{2^{2m}}.  
\end{eqnarray} 
\end{theorem}

\begin{proof}
By the definition of $G$, the code $\C(f_1, \cdots, f_\ell)$ contains the first-order Reed-Muller code 
$\RM_2(1, 2m)$ as a subcode, having weight enumerator (\ref{rm1}).

It follows from (\ref{G}) that every codeword of $\C(f_1, \cdots, f_\ell)$ must be the truth table of a Boolean function of the form   
\begin{eqnarray*}
f_{(u, v, h)}(x)=\sum_{i=1}^\ell u_i f_i(x) + \sum_{j=0}^{2m-1} v_j \Tr_{2m/1}(w^jx) + h,  
\end{eqnarray*} 
where $u_i, v_j, h \in \gf(2)$, $x \in \gf(2^{2m})$. 

Suppose that $F(x)=(f_1(x), f_2(x), \cdots, f_\ell(x))$ is a $(2m,\ell)$
bent vectorial function. 
When $(u_1, \cdots, u_\ell)=(0, \cdots, 0)$, $(v_0, v_1, \cdots, v_{2m-1})$ runs 
over $\gf(2)^{2m}$ and $h$ runs over $\gf(2)$, the truth tables of the functions 
$f_{(u, v, h)}(x)$ form the code $\RM_2(1, 2m)$. Whenever $(u_1, \cdots, u_\ell) \neq 
(0, \cdots, 0)$, it follows from (\ref{G}) that $f_{(u, v, h)}(x)$ is a bent function, and the 
corresponding codeword has Hamming weight $2^{2m-1} \pm 2^{m-1}$. 
Since the all-one vector belongs to $\RM_2(1, 2m)$,
the code $\C(f_1, \cdots, f_\ell)$ is self-complementary, and the desired weight 
enumerator of $\C(f_1, \cdots, f_\ell)$ follows.   

Suppose that $\C(f_1, \cdots, f_\ell)$ has  weight enumerator given by
\eqref{eqn-wtenumerator111}. 
Then $\C(f_1, \cdots, f_\ell)$ has dimension $2m+1+\ell$. Consequently, 
$\sum_{i=1}^\ell u_i f_i(x)$ is the zero function if and only if 
$(u_1, \cdots, u_\ell)=(0, \cdots, 0)$. 
It then follows that the codewords corresponding to $f_{(u, v, h)}(x)$ must have Hamming weight 
$2^{2m-1} \pm 2^{m-1}$ for all $u=(u_1, \cdots, u_\ell) \neq (0, \cdots, 0)$ and 
all $(v_0, v_1, \cdots, v_{2m-1}) \in \gf(2)^{2m}$. Notice that 
$$ 
\sum_{j=0}^{2m-1} v_j \Tr_{2m/1}(w^jx) 
$$ 
ranges over all linear functions from $\gf(2^m)$ to $\gf(2)$ when 
$(v_0, v_1, \cdots, v_{2m-1})$ runs over $\gf(2)^{2m}$.  
Consequently, $F(x)$ is a bent 
vectorial function.     

\end{proof} 

\begin{note}
\label{Note6}
Let $F(x)=(f_1(x), f_2(x), \cdots, f_m(x))$  
be a bent vectorial function from $\gf(2^{2m})$ to $\gf(2)^m$. 
Then the code $\C(f_1, \cdots, f_m)$ has parameters 
$$ [2^{2m}, 3m+1, 2^{2m-1} - 2^{m-1}]. $$
In particular, if $m=2$, any code $\C(f_1, f_2)$
based on a bent vectorial function from $\gf(2^{4})$ to $\gf(2)^2$
 has parameters $[16, 7, 6]$ and is optimal (cf. \cite{MG}).
An $[n,k,d]$ code is optimal if $d$ is the maximum possible
minimum distance for the given $n$ and $k$.
If $m=3$, any code $\C(f_1, f_2, f_3)$
based on a bent vectorial function from $\gf(2^{6})$ to $\gf(2)^3$
 has  parameters $[64, 10, 28]$ and is optimal \cite{MG}. 
If $m=4$, any code $\C(f_1, \cdots, f_6)$
based on a bent vectorial function from $\gf(2^{8})$ to $\gf(2)^4$
 has 
parameters $[256, 13, 120]$ and has the largest known
minimum distance for the given code length and dimension \cite{MG}.   
\end{note}

\begin{theorem}
\label{ex6}
Up to equivalence, there is exactly one $[16,7,6]$ code that 
can be obtained from
a $(4,2)$ bent vectorial function.
\end{theorem}
\begin{proof}
The weight enumerator of the second order Reed-Muller code $\RM_{2}(2,4)$
is given by 
\begin{equation*} 
1+140z^4 + 448z^6 + 870z^8 + 448z^{10} + 140z^{12} + z^{16}.
\end{equation*}
The  truth table  of  a bent function $f$ from $\gf(2^4)$ to $\gf(2)$
is a codeword $c_f$ of
$\RM_{2}(2,4)$ of weight 6. The linear code $\C(f)$ spanned by
$c_f$ and $\RM_{2}(1,4)$ is a subcode of $\RM_{2}(2,4)$
of dimension 6, having  weight enumerator 
\begin{equation*} 
1 + 16z^6 + 30z^8 + 16z^{10} + z^{16}. 
\end{equation*} 
The codewords of $\C(f)$ of weight 6 form a symmetric 2-$(16,6,2)$
SDP design, whose blocks correspond to the supports of 16 bent functions.

Now, let $(f_1,f_2)$ be a $(4,2)$ bent vectorial function. Then,
the intersection of the codes $\C(f_1)$, $\C(f_2)$ consists of the first order
Reed-Muller code $\RM_{2}(1,4)$. It follows that the set of 448
codewords of weight 6 in $\RM_{2}(2,4)$ is a union $\cal{U}$ of
28 pairwise disjoint subsets of size 16, corresponding to the incidence
matrices of symmetric 2-$(16,6,2)$ SDP designs associated 
with 28 different $[16,6]$  codes defined by single bent functions.

If $\C(f_1,f_2)$ is a $[16,7]$ code defined by a bent vectorial
function $(f_1,f_2)$, its weight enumerator is given by 
\begin{equation}\label{f1f2}
1 + 48z^6 + 30z^8 + 48z^{10} +z^{16}.
\end{equation}
The  set of 48 codewords of weight 6 of $\C(f_1,f_2)$ is a union
of the incidence matrices of three  SDP designs from $\cal{U}$
with pairwise disjoint sets of blocks.
A quick 
check  shows that there are exactly 56 such
collections of 48 codewords that generate a code  having weight
enumerator  (\ref{f1f2}). Therefore, the number of distinct
$[16,7,6]$ subcodes of $\RM_{2}(1,4)$ based on $(4,2)$ bent vectorial functions
is 56.
The $7 \times 16$  generator matrix $G$ of one such $[16,7,6]$ code is listed below: 
\begin{eqnarray*}
\left[  
\begin{array}{cccccccccccccccc}
 0& 0& 0& 0& 0& 0& 0& 0& 1& 1& 1& 1& 1& 1& 1& 1 \\
 0& 0& 0& 0& 1& 1& 1& 1& 0& 0& 0& 0& 1& 1& 1& 1 \\ 
 0& 0& 1& 1& 0& 0& 1& 1& 0& 0& 1& 1& 0& 0& 1& 1 \\ 
 0& 1& 0& 1& 0& 1& 0& 1& 0& 1& 0& 1& 0& 1& 0& 1 \\ 
 1& 1& 1& 1& 1& 1& 1& 1& 1& 1& 1& 1& 1& 1& 1& 1 \\ 
 0& 0& 0& 1& 0& 1& 1& 1& 0& 1& 0& 0& 0& 0& 1& 0 \\ 
 0& 0& 0& 0& 0& 1& 0& 1& 0& 0& 1& 1& 0& 1& 1& 0
\end{array}
\right]. 
\end{eqnarray*}
The first five rows of $G$ form a generator matrix of $\RM_{2}(1,4)$,
while the last two rows are codewords of weight 6 in $\RM_{2}(2,4)$.
 The full automorphism group of the $[16,7,6]$ code generated by $G$
is of order 5760. Since the order of the automorphism group of
$\RM_{2}(1,4)$ is 322560, and
\[ 322560/5760 = 56, \]
it follows that all 56 $[16,7,6]$ codes based on $(4,2)$ bent vectorial functions
are pairwise equivalent.
\end{proof}

The next two examples illustrate that there are at least
three inequivalent optimal $[64,10,28]$ codes that are obtainable
from bent vectorial functions from $\gf(2^{6})$ to $\gf(2)^3$.
The parameters  $[64,10,28]$ correspond to $m=3$ in Note \ref{Note6}.

\begin{example}
\label{bch6410}
The binary cyclic $[63,10]$ code $\C$ with parity check polynomial
$h(x)=(x+1)(x^3 + x^2 +1)(x^6 + x^5 + x^4 + x + 1)$
has weight enumerator  
\[
1 + 196z^{27} + 252z^{28} + 63z^{31} +63z^{32}+252z^{35} + 
196z^{36} +z^{63}.
\] 
The $[63,7]$ subcode $\C'$ of $\C$ having check polynomial
$h'(x)=(x+1)(x^6 + x^5 + x^4 + x + 1)$ has weight enumerator
\[
1 + 63z^{31} +63z^{32}+ z^{63}.
\] 
The extended $[64,7]$ code $(\C')^*$ of $\C'$ has weight  enumerator
\[ 1 + 126z^{32} + z^{64}, \]
hence, $(\C')^*$ is equivalent to  the first order Reed-Muller code $\RM_{2}(1,6)$.
The extended $[64,10]$ code $\C^*$ of $\C$ has weight  enumerator
given by 
\begin{equation}
\label{641028}
 1 + 448z^{28} + 126z^{32} + 448 z^{36} + z^{64}. 
 \end{equation}
Since $\C^*$ contains a copy of  the first order Reed-Muller code  $\RM_{2}(1,6)$ as a subcode, it follows from Theorem \ref{thm-bentvectf}
that $\C^*$ can be obtained from a $(6,3)$ bent vectorial function
from $\gf(2^6)$ to $\gf(2^3)$.
The full automorphism group of $\C^*$ is of order
\[677,376 = 2^9 \cdot 3^3 \cdot 7^2. \]
Magma  was used for these computations.
\end{example}

\begin{example}
\label{example9}
Let $M$ be the 7 by 64 $(0,1)$-matrix with the following structure:
the $i$th column of the  6 by 64 submatrix $M'$ of $M$ consisting
of its first six rows is the binary
presentation of the number $i$ ($i =0, 1, \ldots 63$), while
the last row of $M$ is the all-one row. Clearly, $M$ is a generator
matrix of a binary linear $[64,7]$ code equivalent to  the first order Reed-Muller code
$\RM_{2}(1,6)$.

The first six rows of $M$ can be viewed as the truth tables
of the single Boolean variables $x_1, x_2, \dots x_6$, while
the seventh row of $M$ is the truth table of the constant ${\bf 1}$.

We consider the  Boolean bent functions given by 
\begin{eqnarray*}
f_{1}(x_1,\ldots, x_6) & = & x_{1}x_6 + x_{2}x_{5} + x_{3}x_4, \\
f_{2}(x_1,\ldots, x_6) & = & x_{1}x_5 + x_{2}x_4 + x_{3}x_5 + x_{3}x_6,\\
f_{3}(x_1,\ldots, x_6) & = & x_{1}x_4 + x_{2}x_5 + x_{2}x_6 +x_{3}x_4 + x_{3}x_5 + x_{5}x_{6},\\
f_{4}(x_1,\ldots, x_6) & = & x_{1}x_4 +x_{2}x_3 + x_{3}x_6 + x_{5}x_6.
\end{eqnarray*}
The vectorial functions $F_{1}=(f_{1}, f_{2}, f_{3})$, 
$F_{2}=(f_{1}, f_{2}, f_{4})$
give via Theorem \ref{thm-bentvectf} binary linear codes $\C_1, \ \C_2$ 
with parameters $[64,10,28]$,
having  weight enumerator given by (\ref{641028}).  

The automorphism groups of the codes  $\C_1, \ \C_2$
were computed using the computer-algebra package Magma
\cite{magma}.
 
The code $\C_1$ has full automorphism group of order
\[ 10,752 = 2^{9}\cdot 3\cdot 7. \]
The code $\C_2$ has full automorphism group of order
\[ 4,032 = 2^{6}\cdot 3^{2}\cdot 7. \]
Thus, $\C_1$, $\C_2$ and
 the extended cyclic code $\C^*$ from Example \ref{bch6410}
are pairwise inequivalent. 

We note that the code $\C_1$ 
cannot be equivalent to any extended cyclic code
because its group order is not divisible by 63.
\end{example}

\begin{note}
\label{note10}
The full automorphism group of $\C_1$ from Example \ref{example9}
cannot be 2-transitive
because  its order is not divisible by 63.
 Thus, the code $\C_1$ does not satisfy
 the classical sufficient condition to support 2-designs based on the
2-transitivity of its automorphism group (recall that according to
 \cite{K69},
any 2-homogeneous group of degree 64 is necessarily 2-transitive).

In addition, the minimum distance of its dual code
 ${\C_1}^\perp$ is 4, thus the Assmus-Mattson theorem guarantees only
 1-designs to be supported by $\C_1$.
 
We will prove in the next section that all codes obtained from bent vectorial
functions support 2-designs. 
\end{note}

\section{A construction of $2$-designs from bent vectorial functions} 


The following theorem establishes that
the binary codes based on bent vectorial functions
support 2-designs, despite that these codes do not meet 
the conditions of the Assmus-Mattson theorem 
for 2-designs.

\begin{theorem}\label{main}
Let  $F(x)=(f_1(x), f_2(x), \cdots, f_\ell(x))$
be a bent vectorial function from $\gf(2^{2m})$ to $\gf(2)^\ell$,
where $m\ge 2$ and  $1\le \ell \le m$.
Let $\C = \C(f_1, \cdots, f_\ell)$ be the binary linear code
with parameters 
$[2^{2m}, 2m+1 + \ell, 2^{2m-1} - 2^{m-1}]$
defined  in Theorem \ref{thm-bentvectf}.

(a) The codewords of $\C$ of minimum weight hold a 2-design $\bD$
with parameters 
\begin{equation}
\label{par}
2-(2^{2m}, 2^{2m-1} - 2^{m-1}, (2^\ell -1)(2^{2m-2} - 2^{m-1})).
\end{equation}
(b) The codewords of $\C$ of  weight $2^{2m-1} + 2^{m-1}$ hold a 
2-design $\overline{\bD}$ with parameters
\begin{equation}\label{prc}
 2-(2^{2m}, 2^{2m-1} + 2^{m-1}, (2^\ell -1)(2^{2m-2} + 2^{m-1})).
 \end{equation}

\end{theorem}

\begin{proof}
Since $\C$ contains $\RM_2(1,2m)$, and the minimum distance 
of $\RM_2(1,2m)^\perp$ is 4,
the minimum distance $d^{\perp}$ of $\C^{\perp}$ is at least 4. Applying the
MacWilliams transform (see, for example \cite[p. 41]{vanLint}) 
to the weight enumerator (\ref{eqn-wtenumerator111}) of $\C$
shows that $d^{\perp} =4$.
It follows from the Assmus-Mattson theorem (Theorem \ref{thm-AM1}) that the codewords
of any given nonzero weight $w< 2^{2m}$ in $\C$ hold a 1-design.

However, we will prove that $\C$ actually holds 2-designs, despite that
the Assmus-Mattson theorem guarantees only
 1-designs to be supported by $\C$.

Since the subcode $\RM_2(1,2m)$ of $\C$ contains
all codewords of $\C$ of weight $2^{2m-1}$, the codewords of this weight
hold a 3-design $\cal{A}$ with parameters 3-$(2^{2m}, 2^{2m-1}, 2^{2m-2} -1)$.
We note that $\cal{A}$ is a 2-design with
\begin{equation}\label{eq2m}
\lambda_2 = \frac{2^{2m}-2}{2^{2m-1} -2}\cdot(2^{2m-2} -1)=2^{2m-1}-1.
\end{equation}

Let $\bD$ be the 1-design supported by codewords of weight $2^{2m-1} - 2^{m-1}$.
Since the number of codewords of  weight $2^{2m-1} - 2^{m-1}$ is equal to 
$(2^\ell -1)2^{2m}$,
 $\bD$ is a 1-design with parameters 
 1-$(2^{2m}, 2^{2m-1} - 2^{m-1}, (2^\ell -1)(2^{2m-1} - 2^{m-1}))$.

 Every codeword of $\C$ of weight $2^{2m-1} + 2^{m-1}$ is the sum of
a codeword of weight $2^{2m-1} - 2^{m-1}$ and the all-one vector.
Thus, the codewords  of weight $2^{2m-1} + 2^{m-1}$ hold a 1-design $\overline{\bD}$
having parameters
1-$(2^{2m}, 2^{2m-1} + 2^{m-1}, (2^\ell -1)(2^{2m-1} + 2^{m-1}))$.
Clearly, $\overline{\bD}$
is the complementary design of $\bD$, that is, every block of $\overline{\bD}$
is the complement of some block of $\bD$.

Let $M$ be the $2^{2m+1+\ell} \times 2^{2m}$ $(0,1)$-matrix having as rows the codewords
of $\C$. Since $d^{\perp} =4$, $M$ is an orthogonal array of strength 3,
that is, for every integer $i$, $1 \le i \le 3$, and for every set of $i$ distinct columns
of $M$, every binary vector with $i$ components appears exactly
$2^{2m+1+\ell - i}$ times among the rows of the $2^{2m+1+\ell} \times i$ submatrix of $M$
formed by the chosen $i$ columns.
In particular, any $2^{2m+1+\ell} \times  2$ submatrix consisting of two distinct
columns of $M$ contains the binary vector $(1,1)$ exactly $2^{2m+\ell -1}$ times
as a row. Among these $2^{2m+\ell -1}$ rows, one corresponds to the all-one codeword
of $\C$, $2^{2m-1}-1$ rows correspond to
codewords of weight $2^{2m-1}$ (by equation (\ref{eq2m})), and the
remaining
\begin{equation}\label{eqw3}
2^{2m +\ell -1} - 1 - (2^{2m-1}-1)=(2^{\ell} -1)2^{2m-1}
\end{equation}
rows are  labeled by codewords of weight $2^{2m-1} \pm 2^{m-1}$,
corresponding to blocks of $\bD$ and $\overline{\bD}$.

Let now $1\le c_1 < c_2 \le 2^{2m}$ be two distinct columns of $M$.
These two columns label two distinct points of $\bD$ (resp. $\overline{\bD}$).
Let $\lambda$ denote the number of blocks of $\bD$ that are incident with $c_1$
and $c_2$. Then the pair $\{ c_1, c_2 \}$ is incident with
\begin{equation}\label{eqw4}
(2^\ell -1)2^{2m} -2(2^\ell -1)(2^{2m-1}- 2^{m-1}) + \lambda =(2^\ell -1)2^m + \lambda 
\end{equation}
blocks of the complementary design $\overline{\bD}$. 
It follows from (\ref{eqw4})
and (\ref{eqw3}) that
\[ (2^\ell -1)2^m + 2\lambda =(2^{\ell} -1)2^{2m-1},  \]
whence
\[ \lambda = (2^\ell -1)(2^{2m-2} - 2^{m-1}), \]
and the statements  (a) and (b) of the theorem follow. 
\end{proof}

The special case $\ell =1$ in Theorem \ref{main}
implies as a corollary the following result 
of  Dillon and Schatz \cite{DS87}.

\begin{theorem}\label{thm-DSthm}
Let $f(x)$ be a bent function from $\gf(2^{2m})$ to $\gf(2)$. Then the 
code $\C(f)$ has parameters $[2^{2m}, 2m+2, 2^{2m-1} - 2^{m-1}]$ and weight 
enumerator (\ref{sdpcode}). 
The minimum weight codewords form a symmetric SDP design with  parameters 
 (\ref{eqn-SDPparameters}). 
\end{theorem} 

\begin{proof}
The weight enumerator (\ref{sdpcode}) is obtained by substitution $\ell =1$ in
(\ref{eqn-wtenumerator111}). Since the number of minimum weight vectors
is equal to the code length $2^{2m}$, the 2-design   $\bD$ supported
 by the codewords of minimum weight is symmetric. Since every two blocks
 $B_1, B_2$ of $\bD$
 intersect in $\lambda = 2^{2m-2} - 2^{m-1}$ points, 
 the sum of the two codewords
 supporting $B_1$, $B_2$ is a codeword $c_{1,2} $ of weight $2^{2m-1}$ that belongs to
 the subcode $\RM_2(1,2m)$.  
 
 Let $B_3$ be a block distinct from $B_1$ and $B_2$, and
let $c_3$ be the codeword associated with  $B_3$.
Since $c_3$ is the truth table of a bent function, the sum
$c_{1,2}+c_3$ is a codeword of weight $2^{2m-1} \pm 2^{m-1}$,
thus its support is either a block or the complement of a block of $\bD$.
Therefore, $\bD$ is an SDP design.
 
\end{proof} 

\begin{theorem}\label{minw}
The code   
 $\C = \C(f_1, \cdots, f_\ell)$ from Theorem \ref{main}
is spanned by the set of codewords of minimum weight.
\end{theorem}

\begin{proof}
All we need to prove is that the
copy of $\RM_{2}(1,2m)$ which is a subcode of  $\C$,
is spanned by some minimum weight codewords of $\C$.

It is known that the 2-rank 
(that is, the rank over $\gf(2)$) of the incidence matrix
of any symmetric SDP design $\bD$ with $2^{2m}$ points is equal to
$2m+2$ (for a proof, see \cite{JT92}). This implies that
the binary code spanned by $\bD$ contains the first order
Reed-Muller code $\RM_{2}(1,2m)$.
Consequently 
the minimum weight vectors of the 
subcode $\C_{f_1} = \C(f_1)$ of $\C=\C(f_1,\ldots, f_{\ell})$ span the subcode of $\C$ being
equivalent to $\RM(1,2m)$.
\end{proof}

\begin{corollary}\label{thm-equivcodedesign}
Two codes $\C_f =\C(f_1, \cdots, f_s)$,  $\C_g =\C(g_1, \cdots, g_s)$
 obtained from bent vectorial 
functions $F(f_1, \cdots, f_s)$,  $F(g_1, \cdots, g_s)$
 are equivalent if and only if the  designs
supported by their minimum weight vectors  are isomorphic.
\end{corollary}

\begin{example} 
\label{ex15}
Let $m=5$. Let $w$ be a generator of $\gf(2^{10})^*$ with $w^{10} + w^6 + w^5 + 
w^3 + w^2 + w + 1=0$. Let $\beta=w^{2^5+1}$. Then $\beta$ is a generator
of $\gf(2^5)^*$. 
Define $\beta_j=\beta^j$ for $1 \leq j \leq 5$. Then $\{\beta_1, \beta_2, \beta_3, 
\beta_4, \beta_5\}$ is a basis of $\gf(2^5)$ over $\gf(2)$. Now consider the bent 
vectorial function $(f_1, f_2, f_3, f_4, f_5)$ in Example \ref{exam-bentvectfunc1} 
and the code $\C(f_1, f_2, f_3)$. 

When $i=1$ and $i=7$, the two codes $\C(f_1, f_2, f_3)$ have parameters $[1024, 14, 496]$ 
and weight enumerator 
$$ 
1 + 7168z^{496} + 2046z^{512} + 7168z^{528} +z^{1024}. 
$$ 
The two codes are not equivalent according to Magma. 
It follows from Corollary \ref{thm-equivcodedesign} that the two designs with 
parameters $2$-$(1024, 496, 1680)$  supported by these codes
are not isomorphic. 
\end{example} 

\begin{note}
\label{note-conj}
Examples \ref{bch6410} and \ref{example9} give 
three inequivalent $[64,10,28]$ codes,
and Example \ref{ex15}
lists two inequivalent codes with parameters $[1024, 14, 496]$, 
 obtained from bent vectorial functions. 
As we pointed out in Note \ref{note10}, the code 
$\C_1$ from Example \ref{example9}, does not have a 2-transitive group.
\end{note}

These examples, as well as further evidence provided
by Theorem \ref{thnew} below, suggest the following plausible statement  
that we formulate as a conjecture.

\begin{conj}
\label{conjecture}
For any given $\ell$ in the range $1\le \ell \le m$,
the number of inequivalent codes
with parameters $[2^{2m}, 2m+1 +\ell, 2^{2m-1}-2^{m-1}]$
 obtained from $(2m,\ell)$
 bent vectorial functions via Theorem \ref{thm-bentvectf}, 
grows exponentially with linear growth of $m$,
and most
 of these codes
do not admit a 2-transitive automorphism group. 
\end{conj}

As it is customary, by ``most'' we mean that the limit of the ratio of the 
number of 2-transitive codes divided by the total number of 
codes approaches zero when $m$ grows to infinity.

The next theorem proves Conjecture \ref{conjecture} in the case $\ell = 1$.

\begin{theorem}
\label{thnew}
(i) The number of inequivalent $[2^{2m},2m+2,2^{2m-1}-2^{m-1}]$
codes obtained from single bent functions 
from $GF(2^{2m})$ to $GF(2)$ 
grows exponentially with linear growth of $m$.

(ii) For every given $m\ge 2$, there is exactly one (up to equivalence)
code with parameters  $[2^{2m},2m+2,2^{2m-1}-2^{m-1}]$ obtained
from a bent function from $GF(2^{2m})$
to $GF(2)$, that admits a 2-transitive automorphism group.
\end{theorem}
\begin{proof}
(i) By the Dillon-Schatz Theorem \ref{thm-DSthm}, the minimum weight codewords
of a code $\C(f)$ with parameters $[2^{2m},2m+2,2^{2m-1}-2^{m-1}]$ obtained
from a bent function $f$
 form a symmetric SDP design
 $\bD(f)$ with  parameters (\ref{eqn-SDPparameters}).
It follows from Theorem \ref{minw} that two codes $\C(f_1)$, $\C(f_2)$
obtained from bent functions $f_1$, $f_2$
are equivalent if and only if the the corresponding designs
$\bD(f_1)$, $\bD(f_1)$ are isomorphic.
Since the number of nonisomorphic SDP designs with  parameters 
(\ref{eqn-SDPparameters}) grows exponentially when $m$ grows 
to infinity (Kantor \cite{Kantor83}),
the proof of part (i) is complete.

(ii) It follows from Theorem \ref{minw} that the automorphism group
of a code $\C(f)$ obtained from a bent function $f$ coincides
with the automorphism group of the design $\bD(f)$ supported by
the codewords of minimum weight. The design $\bD(f)$ is a symmetric
2-design with parameters (\ref{eqn-SDPparameters}).
It was proved by Kantor \cite{K85-2} that for every $m\ge 2$,
there is exactly one (up to isomorphism) symmetric design
with parameters (\ref{eqn-SDPparameters}) that admits
 a 2-transitive automorphism group. This completes the proof of part (ii).

\end{proof}
 
By Theorem \ref{thm-DSthm}, the codes based on single bent functions
support symmetric 2-designs. The next theorem determines the block intersection
numbers of the design $\bD(f_1, \cdots, f_\ell)$ supported by the minimum weight vectors
in the code $\C(f_1, \cdots, f_\ell)$ from Theorem \ref{main}.

\begin{theorem}\label{thm-int}
Let   $\bD=\bD(f_1,\ldots,f_\ell)$, ($1\le \ell \le m$), be a 2-design
with parameters $$ 2-(2^{2m}, 2^{2m-1} - 2^{m-1}, (2^\ell -1)(2^{2m-2} - 2^{m-1})) $$
supported by the minimum weight codewords of a code $\C =\C(f_1,\ldots,f_\ell)$
defined as in Theorem \ref{main}.

(a) If $\ell =1$, $\bD$ is a symmetric SDP design, with block intersection number
$\lambda = 2^{2m-2} - 2^{m-1}$. 

(b) If  $2 \leq \ell \leq m$,
$\bD$  has the following three block intersection numbers:  
\begin{equation}\label{s123}
s_1 = 2^{2m -2} - 2^{m-2}, \ s_2 =  2^{2m -2} - 2^{m-1}, \ s_3 =
 2^{2m -2} - 3\cdot 2^{m-2}.
\end{equation}
For every block of $\bD$, these intersection numbers occur with multiplicities
\begin{equation}\label{n123}
n_1 =2^{m}(2^m +1)(2^{\ell-1} -1),  \  n_2 = 2^{2m} -1, \   n_3 = 2^{m}(2^m -1)(2^{\ell-1} -1).
\end{equation}
\end{theorem}

\begin{proof}
Case (a) follows from Theorem \ref{thm-DSthm}.

(b) Assume that $2 \leq \ell \leq m$.
Let $w_1$, $w_2$ be two distinct codewords
of  weight $2^{2m-1} - 2^{m-1}$.
The Hamming distance $d(w_1,w_2)$ between $w_1$ and $w_2$
is  equal to
\[ 2(2^{2m-1} - 2^{m-1}) - 2s, \]
where $s$ is the size of the intersection of the supports of $w_1$ and $w_2$.
Since the distance between $w_1$ and $w_2$ is
either $2^{2m-1} - 2^{m-1}$, or $2^{2m-1}$, or
$2^{2m-1} + 2^{m-1}$, the size $s$ of the intersection of
the two blocks of $\bD$ supported by $w_1$, $w_2$
can take only the values $s_i$, $1\le i \le 3$,
given by (\ref{s123}).

Let $B$ be a block of $\bD$ supported by a codeword of weight
$2^{2m-1} - 2^{m-1}$, and let $n_i$, ($1\le i \le 3$),
denote the number of blocks of $\bD$ that intersect $B$
in $s_i$ points. Let ${\bf r} = (2^\ell -1)(2^{2m-1} - 2^{m-1})$ 
denote the number of blocks of $\bD$ containing  a single point, 
and let $b = (2^\ell -1)2^{2m}$ denote the total number of blocks of $\bD$. 
Finally, let $k = 2^{2m-1} - 2^{m-1}$ denote the size of a block, 
and let $\lambda=(2^\ell -1)(2^{2m-2} - 2^{m-1})$ denote the number of blocks
containing two points.
We have
\begin{eqnarray*}
n_1  +  n_2  +  n_3 & = & b - 1,\\
s_{1}n_1  +  s_{2}n_2  +  s_{3}n_3 & = & k({\bf r} - 1), \\
s_{1}(s_{1} - 1)n_1  +  s_{2}(s_{2} -1)n_2 +  s_{3}(s_{3} -1)n_3 & = &
k(k-1)(\lambda -1).
\end{eqnarray*}
The second and the third equation count in two ways the appearances of single points
and ordered pairs of points of $B$ in other blocks of $\bD$.
The unique solution of this system of equations for $n_1,\  n_2, \ n_3$ 
 is given by (\ref{n123}).
\end{proof} 

\begin{note} 
A {\it bent set} is a set  $S$ of bent functions
such that the sum of every two functions from $S$ is 
also a bent function \cite{BK}. Since every $(2m,\ell)$ bent vectorial function
gives rise to a bent set consisting of $2^\ell$ functions \cite[Proposition~1]{BK},
it follows from \cite[Theorem~1]{BK} that the set of blocks of the design $\bD$
is a union of $2^\ell-1$ linked system of symmetric
2-$(2^{2m}, 2^{2m-1}-2^{m-1},2^{2m-2}-2^{m-1})$ designs.
This gives an alternative proof of Theorem~\ref{main} and
Theorem~\ref{thm-int}(b).
\end{note}

\begin{note}
For every integer $m \ge 2$, any  code  $\C(f_1, f_2, \ldots, f_m)$ 
based on a bent vectorial function $F(x)=(f_1(x), f_2(x), \cdots, f_m(x))$
from $\gf(2^{2m})$ to $\gf(2)^m$, 
contains
$2^{m}-1$ subcodes $\C'=\C'(f_{j_1},\ldots, f_{j_s})$,
$j_1 < \cdots < j_s \le m$, such that 
$$ 
\RM_2(1, 2m) \subset \C' \subseteq \C(f_1,  \ldots, f_m). 
$$ 
Each  subcode $\C'$  holds $2$-designs.  This may be the only known chain 
of linear codes, included in each other,  other than the chain of the
Reed-Muller codes,  
\[  \RM_2(1, 2m) \subset  \RM_2(2, 2m) \subset \cdots \subset  \RM_2(m-2, 2m). \]
such that all codes in the chain support nontrivial 2-designs. 
\end{note}

\begin{note}
We would demonstrate that the characterization of bent vectorial functions in Theorem 
\ref{thm-bentvectf} can be used to construct bent vectorial functions. To this end, 
consider the extended binary narrow-sense primitive BCH code of length $2^{2m}-1$ 
and designed distance $2^{2m-1}-1-2^{m-1}$, which is affine-invariant and holds 
$2$-designs \cite{DingZhou}. This code has the desired weight enumerator of (\ref{eqn-wtenumerator111}) for $\ell = m$ \cite{DingZhou}. It can be proved with the 
Delsarte theorem that the trace representation of this code is equivalent to the 
following code: 
\[\left\{\left(f_{a,b,h}(x)\right)_{x\in\gf(2^{2m})}:
  a \in \gf(2^m), \, b \in \gf(2^{2m}), \, h \in \gf(2)  
 \right\},\]  
where
\[f_{a,b,h}(x)=
\tr_{m/1}\left[a \tr_{2m/m}\left(x^{1+2^{m-1}}\right) \right] + \tr_{2m/1}(bx) + h.\]
It then follows from Theorem \ref{thm-bentvectf} that $\tr_{2m/m}(x^{1+2^{m-1}})$ is a 
bent vectorial function from $\gf(2^{2m})$ to $\gf(2^m)$. Note that this bent vectorial 
function may not be new. But our purpose here is to show that bent vectorial functions 
could be constructed from special linear codes.   

Conversely, we could say that the extended narrow-sense BCH code of length $2^{2m}-1$ and  designed 
distance $2^{2m-1}-1-2^{m-1}$ is in fact generated from the bent vectorial function 
$\tr_{2m/m}(x^{1+2^{m-1}})$ from $\gf(2^{2m})$ to $\gf(2^m)$ using the construction 
of Note \ref{note-tracecons}. 

Example \ref{bch6410} gives a demonstration of that. 
Thus, all known binary codes with the weight enumerator 
 (\ref{eqn-wtenumerator111}) for some $1\le \ell \le m$ and arbitrary $m\ge 2$
 are obtained from the bent vectorial function construction. 
As shown in Example \ref{ex6}, all $[16,7,6]$ codes obtained from $(4,2)$
 bent vectorial functions are equivalent.
 Example  \ref{example9} shows that
there are at least three inequivalent $[64, 10, 28]$ binary codes from bent vectorial functions, 
one of these codes being an extended BCH code.
\end{note}

\begin{note} 
It is known that two designs $\bD(f)$ and $\bD(g)$ from two single bent Boolean functions $f$ and $g$ 
on $\gf(2^{2m})$ are isomorphic if and only if $f$ and $g$ are weakly affinely equivalent 
\cite{DS87}. Although the classification of bent Boolean functions into weakly affinely 
equivalent classes is open, 
the results from \cite{Kantor83} and \cite{DS87} imply that the
number of nonisomorphic SDP designs and inequivalent bent functions in $2m$ 
variables grows exponentially with linear growth of $m$.
\end{note} 

\begin{note}\label{note-equivconstruct} 
Two $(n, \ell)$ vectorial Boolean functions $(f_1(x), \cdots, f_\ell(x))$ and 
$(g_1(x), \cdots, g_\ell(x))$ from $\gf(2^n)$ to $\gf(2)^\ell$ are said 
to be \emph{EA-equivalent} if there are an automorphism of $(\gf(2^n), +)$, a 
homomorphism $L$ from $(\gf(2^n),+)$ to $(\gf(2)^\ell, +)$, 
an $\ell \times \ell$ invertible matrix $M$ over $\gf(2)$, 
an element $a \in \gf(2^n)$, and an element $b \in \gf(2)^\ell$ such that 
\[(g_1(x), \cdots, g_\ell(x))= 
(f_1(A(x)+a), \cdots, f_\ell(A(x)+a))M +L(x) +b \]
for all $x \in \gf(2^n)$.  

Let $(f_1(x), \cdots, f_\ell(x))$ and $(g_1(x), \cdots, g_\ell(x))$ be two 
bent vectorial functions from $\gf(2^{2m})$ to $\gf(2)^\ell$. 
We conjecture that the designs $\bD(f_{1}, \cdots, f_{\ell})$ 
and $\bD(g_{1}, \cdots, g_{\ell})$ are isomorphic if and only if 
$(f_1(x), \cdots, f_\ell(x))$ and $(g_1(x), \cdots, g_\ell(x))$ are EA-equivalent. 
The reader is invited to attack this open problem.   
\end{note} 
~\\
Suppose that  $\bD$ is a 2-design with parameters (\ref{par})
obtained from a bent vectorial function
$F(x)=(f_1(x), f_2(x), \cdots, f_\ell(x))$, ($1 \le \ell \le m$), via the construction from Theorem \ref{main}.
Let $\cal{B}$ be the block set of $\bD$. 
If $B$ is a block of $\bD$,  we consider the collection of new blocks $\cB^{de}$
consisting of  intersections $B \cap B'$ such that $B' \in \cal{B}$ and $| B \cap B' |=2^{2m-2} - 2^{m-1}$. 

\begin{theorem}
\label{t25}
For each $B \in \bD$, the incidence structure $(B, \cB^{de})$ 
is a quasi-symmetric design with parameters 
$$ 
2-( 2^{2m-1}-2^{m-1},\, 2^{2m-2}-2^{m-1},\, 2^{2m-2}-2^{m-1}-1) 
$$ 
and intersection numbers $2^{2m-3} - 2^{m-2}$ and   $2^{2m-3} - 2^{m-1}$.
\end{theorem} 

\begin{proof}
By Theorem \ref{thm-int}, there are exactly
$2^{2m} -1$
blocks that intersect $B$ in $2^{2m-2}-2^{m-1}$
points. Together with $B$, these blocks form
a symmetric SDP design $D$ with parameters 2-$(2^{2m}, 2^{2m-1} - 2^{m-1}, 2^{2m-2} - 2^{m-1})$.
The incidence structure  $(B, \cB)^{de}$
is a derived design of $D$.
It was proved in \cite{JT92} that
each derived design
 of a symmetric SDP 2-$(2^{2m}, 2^{2m-1} - 2^{m-1}, 2^{2m-2} - 2^{m-1})$ design
 is quasi-symmetric design with intersection numbers
$2^{2m-3} - 2^{m-2}$ and $2^{2m-3} - 2^{m-1}$,
and having the additional property that the symmetric difference of
every two blocks
is either a block or the complement of a block. 
\end{proof}

\begin{note}\label{note-tracecons} 
Let $m >1$ be an integer. Let $F$ be a bent vectorial function from 
$\gf(2^{2m})$ to $\gf(2^m)$. Let $A$ be a subgroup of order $2^s$ of 
$(\gf(2^m), +)$. Define a binary code by 
\begin{eqnarray*}
\C_{A}:=\{(\tr_{m/1}(aF(x))+\tr_{2m/1}(bx)+c)_{x \in \gf(2^{2m})}: 
a \in A, b \in \gf(2^{2m}), c \in \gf(2)\}. 
\end{eqnarray*} 
It can be shown that $\C_{A}$  
can be viewed as a code $\C(f_{i_1}, \cdots, f_{i_s})$ 
obtained from a bent vectorial function $(f_{i_1}, \cdots, f_{i_s})$.  
\end{note}

\section{Summary and concluding remarks} 

The contributions of this paper are the following. 
\begin{itemize}
\item A coding-theoretic characterization of bent vectorial functions  
      (Theorem \ref{thm-bentvectf}). 
\item A construction of a two-parameter family of four-weight binary linear
 codes with parameters $[2^{2m}, 2m+1+\ell, 
      2^{2m-1}-2^{m-1}]$ for all $1 \leq \ell \leq m$ and all $m\ge 2$,
      obtained from $(2m, \ell)$ bent vectorial 
     functions (Theorem \ref{main}).
      The parameters of these codes appear to be new when $2 \leq \ell \leq m-1$. 
       This family of codes includes  some optimal codes, as well as codes meeting the
        BCH bound.
        These codes do not satisfy the conditions of the Assmus-Mattson
theorem, 
but nevertheless hold $2$-designs. 
It is plausible that most of these codes do not
admit 2-transitive automorphism groups (Conjecture \ref{conjecture}
and Theorem \ref{thnew}).

\item A  new construction of  a two-parameter family of $2$-designs   
with parameters 
\begin{eqnarray}
\label{pm}
2\mbox{--}(2^{2m}, \ 2^{2m-1}-2^{m-1}, \ (2^\ell-1)(2^{2m-2}-2^{m-1})), 
\end{eqnarray}  
and having three block intersection numbers, where $2\le \ell \le m$,
based on bent vectorial functions
(Theorem \ref{main} and Theorem \ref{thm-int}).         
This construction is a generalization of the construction of SDP designs
from single bent functions given in \cite{DS87}. 

\item The  number of nonisomorphic designs with parameters
(\ref{pm}) in the special case when $\ell =1$,
 grows exponentially with $m$
by a known theorem of Kantor \cite{Kantor83}.
It is an interesting open problem to prove
that the number of nonisomorphic designs with parameters
(\ref{pm}) grows exponentially for any fixed $\ell >1$. 
\end{itemize}

Finally, we would like to mention that vectorial Boolean functions were employed
in a different way to construct  binary linear codes  in \cite{TCZ17}.
The codes from  \cite{TCZ17} have different parameters from the codes described in this paper.

\section*{Acknowledgements} 
Vladimir Tonchev  acknowledges partial support  by a Fulbright 
grant, and would like to thank the Hong Kong University of Science 
and Technology for the kind hospitality and support during his visit,
when a large portion of this paper was written. The research of Cunsheng Ding 
was supported by the Hong Kong Research Grants Council, under Grant No. 
16300418.  
The authors wish to thank  the anonymous reviewers for their
valuable comments and suggestions for improving the manuscript.

\end{document}